\documentclass[12pt]{amsart}

\usepackage{amsmath, amssymb,amsmath, graphicx, xypic }
\usepackage[small]{caption}
\usepackage{txfonts}

\numberwithin{equation}{section}

\def\N{\mathbb{N}} 
\def\Z{\mathbb{Z}}
\def\Q{\mathbb{Q}}
\def\R{\mathbb{R}}
\def\K{\mathbb{K}}

\def\mc{\mathcal}

\newtheorem{theorem}{Theorem}[section]
\newtheorem{lemma}[theorem]{Lemma}
\newtheorem{proposition}[theorem]{Proposition}
\newtheorem{corollary}[theorem]{Corollary}
\newtheorem{definition}[theorem]{Definition}

\newtheorem{remark}[theorem]{Remark}

\newtheorem{question}[theorem]{Question}

\DeclareMathOperator{\area}{\mathsf{Area}}

\begin{document}

\title[]{A note on  fine graphs and homological isoperimetric inequalities}
\author[E.~Mart\'inez-Pedroza]{Eduardo Mart\'inez-Pedroza}
  \address{Memorial University\\ St. John's, Newfoundland, Canada A1C 5S7}
  \email{emartinezped@mun.ca}

\maketitle

\begin{abstract}
In the framework of homological characterizations of relative hyperbolicity,  Groves and Manning posed the question of whether  a simply connected $2$-complex $X$ with a  linear homological  isoperimetric inequality, a bound on the length of attaching maps of $2$-cells  and finitely many $2$-cells adjacent to any edge must have a fine $1$-skeleton.  We provide a positive answer to this question.  We revisit a homological characterization of relative hyperbolicity, and show that a group $G$ is hyperbolic relative to a collection of subgroups $\mc P$ if and only if $G$ acts cocompactly with finite edge stabilizers on an connected $2$-dimensional cell complex with a linear homological isoperimetric inequality and $\mc P$ is a collection of representatives  of conjugacy classes of vertex stabilizers.
\end{abstract}

\section{Introduction}

In this article, we investigate  the relation between the notion of fine graph and homological isoperimetric inequalities of combinatorial complexes. 
We work in the category of combinatorial complexes and combinatorial maps as defined, for example, in~\cite[Chapter I.8, Appendix]{BrHa99}.  All group actions on complexes are by combinatorial maps.

The notion of fine graph was introduced by Bowditch in his investigations on the theory of relatively hyperbolic groups~\cite{Bo12}. 

\begin{definition}[Fine graph]
A \emph{graph} $\Gamma$ is a $1$-dimensional combinatorial complex.   A \emph{circuit} is a simple closed combinatorial path. A graph $\Gamma$ is \emph{fine} if for every edge $e$ and each integer $L>0$, the number of circuits of length at most $L$ which contain $e$ is finite.  
\end{definition}

Let $\K$ denote either $\Z$, $\Q$ or $\R$.   For a cell complex $X$, the cellular chain group $C_i(X, \K)$  is a free $\K$-module  with a natural $\ell_1$-norm induced by a basis formed by the collection of all $i$-dimensional cells of $X$, each cell with a chosen orientation from each pair of opposite orientations. This norm, denoted by $\|\gamma\|_1$, is the sum of the absolute value of the coefficients in the unique representation of the chain $\gamma$ as a linear combination over $\K$ of the elements of the basis. 

\begin{definition}[Homological Dehn function of a cell complex]\label{def:FVX}   The \emph{homological Dehn function of a cell complex $X$ over $\K$} is the function $FV_{X, \K} \colon\N \to \K \cup \{\infty\}$ defined as 
\[FV_{X, \K} (k) = \sup \left \{ \ \| \gamma   \|_{\partial}  \colon \gamma \in Z_1(X, \Z), \ \| \gamma \|_1 \leq k \ \right \},\]
where 
\[ \| \gamma  \|_{\partial, \K}  = \inf \left \{ \ \| \mu \|_1 \colon \mu \in C_{2}(X, \K), \ \partial ( \mu ) = \gamma \ \right \},\]
where the supremum and infimum of the empty set are defined as zero and $\infty$ respectively. In words,  $FV_{X, \K}(k)$ is the most efficient upper bound on the size of fillings by $2$-chains over $\K$ of $1$-cycles over $\Z$ of size at most $k$.
\end{definition}

 The following result  exhibits the natural relation between the notions of fine graph and homological Dehn function in the context of $G$-spaces.  Observe that a necessary condition for $FV_{X, \K}$ being finite-valued  is that   $X$ has  trivial first homology group over $\K$.  

\begin{theorem}\label{prop:main}
Let $X$ be a cocompact $G$-cell complex with finite stabilizers of $1$-cells.  The following two statements are equivalent:
\begin{enumerate}
\item  $X$ has fine $1$-skeleton and $H_1(X, \Z)$ is trivial,
\item \label{main2} $FV_{X, \Z}(k)<\infty$ for any integer $k$. 
\end{enumerate}
\end{theorem}

\begin{definition}[Homological isoperimetric inequalities]\label{def:hom-isop-ineq}
Let $X$ be a  complex. We  shall say that $X$ satisfies a \emph{homological isoperimetric inequality over $\K$} if  $FV_{X, \K}(k)<\infty$ for any integer $k$, and we say that $X$ satisfies a \emph{linear homological isoperimetric inequality over $\K$} if  there is a constant $A\geq 0$ such that $FV_{X, \K}(k)\leq kA$.
\end{definition}

The definition of linear homological isoperimetric inequality  above is equivalent to the definition used by Groves and Manning~\cite[Definition 2.28]{GrMa09}, see Proposition~\ref{prop:definitions}. The following question was raised in~\cite{GrMa09}.

\begin{question}\label{question} \cite[Question. 2.51]{GrMa09}
 Let $X$ be a simply connected $2$-complex with a homological (linear?) isoperimetric inequality, a bound on the length of attaching maps of $2$-cells  and finitely many $2$-cells adjacent to any edge. Must $X$ be fine?
\end{question}

Question~\ref{question}   was raised in the context of homological isoperimetric inequalities over the rational numbers. It can also be interpreted in the context of homological isoperimetric inequalities over the integers. In both cases the question is answered in the positive by Theorem~\ref{thm:answer} below, but we remark that in the rational case our argument requires the suggested hypothesis of a linear isoperimetric inequality.

\begin{theorem}\label{thm:answer} Let $X$ be a cell complex such that each $1$-cell  is adjacent to finitely many $2$-cells. The $1$-skeleton of $X$ is a fine graph if either
\begin{enumerate}
\item $FV_{X, \Z}(k)<\infty$ for any integer $k$, or 
\item \label{answer2}  $X$ is  simply-connected, there is a bound on the length of attaching maps of $2$-cells, and there is $C\geq 0$ such that  $FV_{X, \Q}(k)\leq Ck$ for every $k$. 
\end{enumerate}
\end{theorem}

The question of whether in Theorem~\ref{thm:answer}, for the rational case,  the assumption  $FV_{X, \Q}(k)<\infty$ for every $k$ is sufficient to conclude fineness remains open. 
In this regard, there is a related question raised by Gersten of whether  there is a constant $C\geq 0$ such that $ FV_{X, \Z}(k) \leq C\cdot FV_{X, \Q}(k)$, see~\cite[Section 4, open question]{GerstenCohomology} and~\cite[Introduction]{Ge99}. A positive answer to Gersten's question would imply that in Theorem~\ref{thm:answer}, for the rational case, only  the assumptions that $X$ is simply-connected and $FV_{X, \Q}(k)<\infty$ are sufficient to conclude fineness.
 
We provide a proof of the following  converse of Theorem~\ref{thm:answer}\eqref{answer2} in the class of cocompact $G$-spaces. For a definition of hyperbolic graph we refer the reader to~\cite{Bo12, BrHa99}. We shall say that a complex $X$ is \emph{$1$-acyclic} if it is connected and has trivial first homology group over the integers.  

\begin{theorem}  \label{thm:forfuture}
Let $G$ be a group. Let $Y$ be a $1$-acyclic cocompact $G$-complex with fine and hyperbolic $1$-skeleton and finite $G$-stabilizers of $1$-cells.  Then there is $C\geq 0$ such that  $FV_{Y, \Z}(k)\leq Ck$ for every $k$. In particular, $FV_{Y, \Q}(k) \leq Ck$ for every $k$.
\end{theorem}

There are results implying hyperbolicity with assumptions in terms of homological linear isoperimetric inequalities over $\Q$ or $\R$. These  are more subtle results.  In the case that $X$ is the universal cover of a $K(G, 1)$ with finite $2$-skeleton and $FV_{X, \K}$ is linearly bounded, Gersten proved that $FV_{X, \Z}$ is also linearly bounded using constructions by Papasoglu and Ol'shanskii, see~\cite[Theorems 5.1 and 5.7]{GerstenCohomology} and the references there in. This argument was revisited by Mineyev in~\cite[Theorem 7]{Mi02}.  Groves and Manning remarked that these arguments do not rely in the complex $X$ being locally finite, and observed that the following result holds. 

\begin{theorem}\label{prop:hyperbolicity}\cite[Theorem 2.30]{GrMa09}
Let $X$ be a simply-connected complex such  there is a  bound on the  length of attaching maps of $2$-cells.
If $FV_{X, \Q}$ is bounded by a linear function  then the $1$-skeleton of $X$ is a hyperbolic graph. 
\end{theorem}

The type of homological functions of Definition~\ref{def:FVX} have been considered in the contexts of  relatively hyperbolic groups for example in~\cite{GrMa09, MP15, MiYa}. Combining  Theorems~\ref{thm:answer},~\ref{thm:forfuture} and~\ref{prop:hyperbolicity} allows us to provide a characterization of  relatively hyperbolic groups in terms of homological Dehn functions stated as   Theorem~\ref{thm:chrhyp} below.  This characterization  resembles the approach to relative hyperbolicity by Osin in terms of a relative Dehn function~\cite{Os06},  strengthens the homological characterization by Groves and Manning~\cite[Theorem 3.25]{GrMa09}, and  extends a characterization of hyperbolic groups by Gersten~\cite[Theorem 3.1]{Ge96}.

We use the definition of relatively hyperbolic groups by Bowditch in terms of cocompact actions on fine graphs~\cite{Bo12}. This approach is equivalent to the well known definitions by Gromov~\cite{Gr87} and Osin~\cite{Os06}  when one restricts to the class of finitely generated groups, see~\cite{HK08, Os06}. 

\begin{definition}[Relatively hyperbolic group]\label{def:BowRH}\cite{Bo12}
A group $G$ is \emph{hyperbolic relative to a finite collection of subgroups $\mc P$} if
$G$ acts   on a connected, fine, $\delta$-hyperbolic graph $\Gamma$ with finite edge stabilizers, finitely many orbits of edges,  and $\mc P$ is a set of representatives of distinct conjugacy classes of vertex stabilizers (such that each infinite stabilizer is represented).   A $G$-graph $\Gamma$ with all these properties is called a $(G, \mc P)$-graph.  
\end{definition}

We remark that in Definition~\ref{def:BowRH} the group $G$ is not assumed to be finitely generated, and there are no assumptions on the subgroups in  $\mc P$.  Recall that a complex $X$ is \emph{$1$-acyclic} if it is connected and has trivial first homology group over the integers.  
 
\begin{theorem}[Relative hyperbolicity characterization]\label{thm:chrhyp}
Let $G$ be a group and let $\mc P$ be a finite collection of subgroups.  Then $G$ is hyperbolic relative to $\mc P$ if and only if there is an $1$-acyclic $G$-complex $X$ such that
\begin{enumerate}
\item the $G$-action on $X$ is cocompact,
\item there is $C\geq 0$ such that  $FV_{X, \Z}(k)\leq Ck$ for every $k$., 
\item the $G$-stabilizers of $1$-cells of $X$ are finite, and 
\item $\mc P$ is a collection of representatives of conjugacy classes of $G$-stabilizers of $0$-cells  such that each infinite stabilizer is represented. 
\end{enumerate}
\end{theorem}

A complex studied in the context of relatively hyperbolic groups is the coned-off Cayley complex $\widehat C$ of a finite presentation of $G$ relative to a collection of finitely generated subgroups $\mc P$, for a definition of this complex see~\cite[Definition 2.47]{GrMa09} or the last section of this note. The statement of Theorem~\ref{thm:chrhyp} replacing $X$ by the coned-off Cayley complex is a homological characterization of relative hyperbolicity by Groves and Manning~\cite[Theorem 3.25]{GrMa09}. Finding a more direct proof of this characterization was one of the motivations of Question~\ref{question}.  A  precise statement  of this  characterization together with a discussion of its proof is in Section~\ref{sec:3.3}.  

The rest of the article is organized in two parts.  The first section contains results on the relation between fine graphs and homological Dehn functions and, in particular,  
the proof of  Theorem~\ref{prop:main}. The second section contains the proofs of Theorems~\ref{thm:answer}, ~\ref{thm:forfuture} and~\ref{thm:chrhyp}; this part concludes with a discussion of coned-off Cayley complexes.

\section{Fine graphs and The homological Dehn function over $\Z$}

Through this article, when considering a complex $X$,   we assume that for each cell of positive dimension an orientation has been chosen once and for all.  
As usual,  the group of $n$-cycles $C_n(X, \Z)$  is understood as the free abelian group with free basis the collection of $n$-cells with their chosen orientation. These chosen orientations are necessary in order to define the boundary maps.   

In this section, we only consider homological Dehn functions over $\Z$, so through all the section  $FV_{X}$ and $\|\cdot\|_\partial$ shall denote $FV_{X, \Z}$ and $\|\cdot \|_{\partial, \Z}$. For statements of results we use the standard notation. 

\subsection{Proof of Theorem~\ref{prop:main}}

\begin{proposition}\label{prop:fine-crit}
Let $X$ be a  complex such that $FV_{X, \Z}(k)<\infty$ for every integer $k$, and each $1$-cell of $X$ is adjacent to finitely many $2$-cells.  Then the $1$-skeleton of $X$ is a fine graph.
\end{proposition}

For the proof of Proposition~\ref{prop:fine-crit}, we  introduce the notions of  \emph{disjoint $1$-chain} and \emph{special $2$-chain}. 

\begin{definition}[Disjoint chains]
Let $X$ be a complex and consider the free abelian group of chains $C_1(X, \Z)$ with basis   the collection of $1$-cells of $X$.  Two  $1$-chains $\alpha, \beta \in C_1(X, \Z)$ are \emph{disjoint} if,  when considering their unique expressions as linear combinations in the basis,  there is no element of the basis having non-zero coefficients in both expressions.   	
\end{definition}

\begin{lemma}\label{lem:key-observation}
Let $\gamma \in Z_1(X, \Z)$ be a cellular $1$-cycle induced by a circuit in the $1$-skeleton of $X$.
If $\gamma = \alpha+\beta$ where $\alpha, \beta \in Z_1(X, \Z)$ are disjoint  then either $\alpha$ or $\beta$ is  trivial.
\end{lemma}
\begin{proof}
Let $\{e_i\}_{i\in I}$ be the collection of $1$-cells of $X$. Suppose that $\gamma = \sum_{i\in I} c_i  e_i$. Since $\gamma$ is induced by a circuit, observe that for every proper subset $J \subsetneq I$ we have that either $\sum_{i\in J} c_i  e_i=\gamma$, or  $\sum_{i\in J} c_i  e_i = 0$, or $\partial \left( \sum_{i\in J} c_i  e_i \right) \neq 0$. Therefore, if $\gamma = \alpha+\beta$ where $\alpha, \beta \in Z_1(X, \Z)$ are disjoint, then either $\alpha=0$ or $\beta=0$.
\end{proof}

\begin{definition}[Special $2$-chain]
Let $X$ be a complex and consider $C_1(X, \Z)$ and $C_2(X, \Z)$ with their free $\Z$-bases corresponding to the collections of  $1$-cells and $2$-cells of $X$ respectively.   Let $e$ be a $1$-cell of $X$.  A \emph{special $2$-chain based at $e$} is a $2$-chain $\mu$ such that there is a sequence  $f_1,  \ldots ,   f_n$ of elements of the basis of $C_2(X, \Z)$ such that 
 $\mu = \sum_{i =1}^n \epsilon_i  f_i$ where  $\epsilon_i= \pm 1$ and 
\begin{enumerate}
\item $\|\mu\|_{1}=n$,
\item the $1$-chains $e$ and $\partial f_1$ are not disjoint, and 
\item for every $k<n$ the $1$-cycles $\partial \sum_{i=1}^k \epsilon_i f_i$ and $\partial f_{k+1}$ are not disjoint. 
\end{enumerate}
 \end{definition}

\begin{remark}
If  $\mu = \sum_{i =1}^n \epsilon_i  f_i$ is a special $2$-chain of $X$ based at $e$, then for every $k\leq n$, the chain $\sum_{i =1}^k \epsilon_i   f_i$ is special.
\end{remark}

\begin{proof}[Proof of Proposition~\ref{prop:fine-crit}] 
Consider $C_1(X, \Z)$ and $C_2(X, \Z)$ with their free $\Z$-bases corresponding to the collections of  $1$-cells and $2$-cells of $X$ respectively.
  We will show that for any  $1$-cell $e$,  any circuit $\gamma$ containing $e$ is (as a $1$-cycle) the boundary of a special $2$-chain $\mu$ based at $e$ such that $\|\mu\|_1\leq FV_X (\|\gamma\|_1)$, this is Claim 1 below. Since $FV_X$ is finite-valued,  it follows  that it is enough to prove that for each positive integer $n$ and each $1$-cell $e$ of $X$, there are finitely many special $2$-chains based at $e$ with $\ell_1$-norm bounded from above by $n$, this is Claim 2 below.

\emph{Claim 1, minimal area fillings are special.} Let $\gamma$ be a $1$-cycle induced by a circuit in the $1$-skeleton of $X$ containing $e$, and let  $\mu$ be a $2$-chain such that  $\partial \mu=\gamma$ and $\|\mu\|_1=\|\gamma\|_\partial $. Then $\mu$ is a special $2$-chain  based at $e$, and in particular $\|\mu\|_1\leq FV_X(\|\gamma\|_1)$.

Indeed, we have a unique expression $\mu = \sum_{i \in I} \epsilon_i   f_i$ where each $f_i$  is an element of the basis of $C_2(X, \Z)$,  $\epsilon_i= \pm 1$, and $\|\mu\|_1$ equals the cardinality of $I$.  Consider a non-empty proper subset of $J\subsetneq I$ and consider the $1$-cycles $\alpha=\partial \left( \sum_{i\in J} \epsilon_i   f_i \right)$ and $\beta = \partial \left ( \sum_{i\in I\setminus J} \epsilon_i    f_i \right )$.  Since $\|\mu\|_1=\|\gamma\|_\partial$, we have that $\alpha$ and $\beta$ are non-zero cycles. Since $\gamma=\alpha+\beta$ is a $1$-cycle induced by a circuit, Lemma~\ref{lem:key-observation} implies that  $\alpha$ and $\beta$ are not disjoint.  An induction argument then shows that we can order $I=\{1,\cdots, n\}$ so that $\mu = \sum_{i =1}^n \epsilon_i  f_i$, the $1$-chains $e$ and $\partial  f_1$ are not disjoint, and  for every $k<n$ the $1$-cycles $\sum_{i=1}^k \epsilon_i   f_i$ and $\partial   f_{k+1}$ are not disjoint. 

\emph{Claim 2.} Let $n$ be a positive integer and let $e$ be a $1$-cell of $X$.  Then there are finitely many special $2$-chains based at $e$ with $\ell_1$-norm equal $n$

Now we use the hypothesis that  each $1$-cell of $X$ is adjacent to finitely many $2$-cells.  Let $\sum_{i =1}^n \epsilon_i   f_i$ be a special $2$-chain based at $e$.  By the hypothesis, there are finitely many choices for $f_1$. Once we have chosen $\sum_{i =1}^k \epsilon_i  f_i$ special based at $e$, since $\partial \left ( \sum_{i =1}^k \epsilon_i   f_i \right)$ and $\partial   f_{k+1}$ are not disjoint, the hypothesis implies that there are finitely many choices for $f_{k+1}$.  
\end{proof}

\begin{proposition}\label{lem:criterionFV}
Let $X$ be a cocompact $G$-complex with trivial first homology and fine $1$-skeleton. Then $FV_{X, \Z}(k)<\infty$ for every integer $k$.  
\end{proposition}

\begin{lemma}\cite[Lemma A2]{Ge98}\label{lem:circuitd}
Let $X$ be a complex. Any $1$-cycle $\gamma \in Z_1(X, \Z)$ can be expressed as a finite sum $\sum_{i} \alpha_i$ where each $\alpha_i$ is a $1$-cycle induced by a circuit and $\|\gamma\|_1 = \sum_{i} \|\alpha_i\|_1$.
\end{lemma}

\begin{proof}[Proof of Proposition~\ref{lem:criterionFV}] 
 Since the $1$-skeleton of $X$ is fine graph and $G$ acts cocompactly, then for each positive integer $n$, the $G$-action on the collection of circuits in the $1$-skeleton of $X$ of length at most $n$ has finitely many orbits. 

Observe the induced actions of $G$ on the cellular chain groups $C_i(X)$ preserve the $\ell_1$-norm and commute with the boundary maps.  In particular, the norm $\|\cdot\|_\partial$ on $Z_1(X, \Z)$ induced by $C_2(X, \Z)$ is $G$-equivariant. 

Since $X$ has trivial first homology and has finitely many circuits of length at most $n$ in the $1$-skeleton up to the $G$-action, there exists a constant $B_n<\infty$ with the following property: $\|\alpha\|_\partial \leq B_n$ for every $1$-cycle $\alpha$ such that $\|\alpha\|_1\leq n$  and $\alpha$ is represented by a circuit of length at most $n$ in the $1$-skeleton of $X$.

Let $\gamma \in \Z_1(X)$ be a cellular $1$-cycle of $X$ such that $\|\gamma\|_1\leq n$. Invoke Lemma~\ref{lem:circuitd} to have an expression  $\gamma=\gamma_1+\gamma_2+\cdots+\gamma_k$ where each $\gamma_i$ is a $1$-cycle represented by a circuit and such that  $\|\gamma\|_1=\|\gamma_1\|_1+ \cdots + \|\gamma_k\|_1$ and  $k\leq n$. Observe that $\|\gamma\|_\partial \leq \|\gamma_1\|_\partial+ \cdots + \|\gamma_k\|_\partial$. It follows that $\|\gamma\|_\partial \leq nB_n$ and hence $FV_X(n) \leq n B_n$.
\end{proof}

\begin{proof}[Proof of Theorem~\ref{prop:main}]
Observe that if $X$ is a cocompact $G$-cell complex with finite stabilizers of $1$-cells, then each $1$-cell is adjacent to finitely many $2$-cells.  The result  follows  from Proposition~\ref{lem:criterionFV} and Proposition~\ref{prop:fine-crit}.  
\end{proof}

\subsection{Isoperimetric functions and $FV_{X, \Z}$}

We refer the reader to~\cite[Appendix: Combinatorial 2-Complexes]{BrHa99} for a discussion on van Kampen diagrams which are used below.

 \begin{definition}[Isoperimetric function]
A function $f\colon \N \to \N$ is an \emph{isoperimetric function} for a complex $X$ if it is monotonic non-decreasing and whenever $P$ is a closed edge path  of $X$, there is van Kampen diagram $D$ for $P$ with area bounded from above by $f(|P|)$, where $|P|$ denotes the combinatorial length of the path. 
\end{definition}  
 
 \begin{definition}[Superadditive closure]
 A function $f\colon \N \to \N$ is superadditive if $f(m)+f(n)\leq f(m+n)$ for every pair $m,n\in \N$. For an arbitrary function $g\colon \N \to \N$, let $\bar g$ denote the least function such that $g\leq \bar g$ and $\bar g$ is super-additive. Specifically, 
 \[\bar g(n) = \max \{f(n_1)+\cdots +f(n_k) \colon n_1+n_2+\cdots +n_k=n \},\]
 where the maximum is taken over all $k\leq n$ and all partitions $n_1+n_2+\cdots +n_k$ of $n$. We shall refer to $\bar g$ as the \emph{superadditive closure} of $g$.  
 \end{definition}

 \begin{remark}\label{rem:linearclosure}
If $f(n)=Cn$ then $\bar f (n) =Cn$.
\end{remark}
   
\begin{proposition}\cite[Proposition 2.4]{Ge99}\label{prop:finiteFV}
Let $X$ be a simply-connected complex admitting an isoperimetric function $f\colon \N \to \N$. Then $FV_{X, \Z}(n) \leq \bar f(n)$ for every $n\in \N$, where $\bar f$ is the superadditive closure of $f$.
\end{proposition}
\begin{proof}
Let  $\gamma \in Z_1(X)$ be a $1$-cycle in $X$ such that $\|\gamma\|_1=n$. By Lemma~\ref{lem:circuitd}, there is an expression   $\gamma=\gamma_1+\cdots+\gamma_k$ where each $\gamma_i$ is a $1$-cycle represented by  a closed path $P_i$ such that  $\|\gamma\|_1=\|\gamma_1\|_1+ \cdots + \|\gamma_k\|_1$ and $\|\gamma_i\|_1=|P_i|$.  For each $i$,  there is a van Kampen diagram $D_i$ with boundary path $P_i$. Observe  the diagram $D_i$ induces a $2$-chain $\mu_i$ such that  $\partial \mu_i =  \gamma_i$.  Since 
$\|\gamma_i\|_\partial \leq \|\mu_i\|_1 \leq \area (D_i) \leq f(|P_i|) = f(\|\gamma_i\|_1)$,  
and
$\|\gamma\|_\partial \leq \|\gamma_1\|_\partial+ \cdots + \|\gamma_k\|_\partial,$
we have  
$|\gamma|_\partial \leq \bar f(n)$. Therefore $FV_X(n) \leq \bar f(n)$.
\end{proof}

The following proposition is a version of the statement that hyperbolicity in terms of thin triangles implies a linear isoperimetric inequality. 

\begin{proposition}\cite[Proposition 3.1]{Bo12} 
\label{lem:simply-connected-complex}
Let $\Gamma$ be a hyperbolic graph with hyperbolicity constant $k$. Then there is a constant $n=n(k)$ with the following property. If $\Omega_n(\Gamma)$ is the $2$-complex with $1$-skeleton the graph $\Gamma$ and such that each circuit of length at most $n$ is the boundary of a unique $2$-cell, then $\Omega_n(\Gamma)$ is simply-connected and admits a linear isoperimetric function. 
\end{proposition}

\subsection{Barycentric subdivisions, fineness, and and $FV_X$}

\begin{lemma}\label{lem:subdivision}
Let $X$ be complex with a bound on the length of attaching maps of $2$-cells, and let $Y$ be the barycentric subdivision of $X$.  Then there is a constant $B=B(X)$ such that
$FV_{X, \Z}(n) \leq FV_{Y, \Z}(Bn)$ and  $FV_{Y, \Z}(n)\leq  B\cdot FV_{X, \Z}(Bn) + Bn$
 for every integer $n$.
\end{lemma}
\begin{proof}[Sketch of a proof]
Denote by $X'$ the cell-complex obtained subdividing each $1$-cell of $X$ into two $1$-cells by inserting an extra $0$-cell at the ``midpoint'' of each $1$-cell, and let $X''$ denote the barycentric subdivision of $X$. One verifies that $FV_{X''}(n)\leq C \cdot FV_{X'}(2Cn)+ 2Cn,$
where $C$ is the maximal length of the boundary path of a $2$-cell in $X$, from which follows that 
$FV_{X''}(n) \leq C\cdot FV_{X}(4Cn)+2Cn.$ Analogously one can show that $FV_X(n)\leq FV_{X''}(2n).$ 
\end{proof}

\begin{lemma}\cite[Lemma 2.9]{MW11}\cite[Lemma 2.4]{Bo12} \label{lem:gsubdivision} 
Let $X$ be a cocompact $G$-complex with fine $1$-skeleton and finite edge stabilizers. Then the $1$-skeleton of its barycentric subdivision  has fine $1$-skeleton.
\end{lemma}
\begin{proof}
 Observe that the barycentric subdivision of a fine graph is fine. Moreover,  the $1$-skeleton of the barycentric subdivision of $X$ is obtained from the barycentric subdivision of the $1$-skeleton of $X$ after $G$-equivariantly attaching finitely many orbits of new arcs (the half-diagonals or diagonals of  higher dimensional cells). This type of construction  was explicitly shown to preserve fineness in~\cite[Lem. 2.9]{MW11}; alternatively it also follows from~\cite[Lem. 2.4]{Bo12}.  
\end{proof}

\section{Homological Dehn Functions and Relative Hyperbolicity}

\subsection{Linear isoperimetric inequalities and hyperbolicity}

 \begin{definition}[Condition $FZ_N$ and Weak linear isoperimetric inequality]\cite[Def. 6.1]{Ge96}\label{def:Gersten}
Let $\Gamma$ be a graph. For an  integer $N$, we shall say that $\Gamma$ satisfies condition   $FZ_N$ if for any circuit $\gamma$ in $\Gamma$ there are circuits $\gamma_1, \gamma_2, \ldots , \gamma_k$ each of length at most $N$ such that 
\begin{equation}\label{eq:FZ}  [\gamma] =  \sum_{i=1}^k \epsilon_i [\gamma_i] \end{equation}
where $[\gamma]$ denotes the class of $\gamma$ in $H_1 (\Gamma, \Z)$ and $\epsilon_i=\pm 1$.  If $\Gamma$ satisfies $FZ_N$, then the \emph{weak area} of the circuit $\gamma$ is the minimum $k$ in all expressions~\eqref{eq:FZ}.  The graph $\Gamma$ satisfies a \emph{weak linear isoperimetric inequality} if there are   integers $N$ and $C$ such that $\Gamma$ satisfies $FZ_N$ and the weak-area of each circuit $\gamma$ is at most $C |\gamma|$ where $|\gamma|$ denotes the length of the circuit.
\end{definition}

The following theorem is a  version by Gersten of the fact that a (standard) linear isoperimetric inequality implies hyperbolicity.  

\begin{theorem}\cite[Thm. 6.3]{Ge96}\label{thm:weak-isop}
If $\Gamma$ is a connected graph satisfying $FZ_N$ and a weak linear isoperimetric inequality then $\Gamma$ is a hyperbolic graph.
\end{theorem}

\begin{corollary}\label{cor:suskey}
Let $X$ be a $1$-acyclic $2$-dimensional cell complex such that there is a bound on the length of  attaching maps of $2$-cells. If $FV_{X, \Z}$ is linearly bounded then the $1$-skeleton of $X$ is a fine hyperbolic graph.\end{corollary}
\begin{proof}
The assumption that there is a bound on the length of the attaching maps of $2$-cells implies that the $1$-skeleton of $X$ and the one of its barycentric subdivision are quasi-isometric. Since hyperbolicity is invariant under quasi-isometry, in view of Lemma~\ref{lem:subdivision},  we can replace $X$ with its  barycentric subdivision and assume that the attaching maps of $2$-cells are circuits.   Let $\gamma$ be a circuit in the $1$-skeleton of $X$. Abusing notation we denote by $\gamma$ the induced $1$-cycle.  Since $X$ has trivial first homology, there is a $2$-chain $\beta \in C_2(X)$ such that $\partial \beta =  \gamma$ and $\|\gamma\|_\partial = \|\beta\|_1$.  Let $N$ be an upper bound for the length of boundary paths of  $2$-cells of $X$, which are assumed to be circuits. It follows that $\gamma = \partial \beta = \sum_{i=1}^m \epsilon_i  \gamma_i$ where each $\gamma_i$ is a  $1$-cycle  induced by a circuit of length at most $N$, and $m=\|\beta\|_1$. It follows that 
$m=\|\beta\|_1 \leq FV_{X, \Z}(\|\gamma\|_1)\leq C \|\gamma\|_1$, 
where $C$ depends only on $X$. Therefore  the $1$-skeleton of $X$ is $FZ_N$ and satisfies a weak linear isoperimetric inequality. By Theorem~\ref{thm:weak-isop}, the  $1$-skeleton of $X$ is a hyperbolic graph.
Since $FV_X$ is finite-valued, Proposition~\ref{prop:fine-crit} implies that the $1$-skeleton of $X$ is fine.
\end{proof}

\subsection{Proof of  Theorem~\ref{thm:answer}}

\begin{proof}[Proof of Theorem~\ref{thm:answer}]
The first statement of the theorem is Proposition~\ref{prop:fine-crit}. 
For the second statement, the assumptions together with Theorem~\ref{prop:hyperbolicity} imply that the $1$-skeleton of $X$ is hyperbolic. It follows that $X$ admits a linear isoperimetric inequality in the standard sense, this follows for example from~\cite[Ch.III.H Proposition 2.2]{BrHa99} or Proposition~\ref{lem:simply-connected-complex}. Then Proposition~\ref{prop:finiteFV} together with Remark~\ref{rem:linearclosure} imply that $FV_{\Z, X}$ is bounded by a linear function, and hence Proposition~\ref{prop:fine-crit} implies that the $1$-skeleton of $X$ is fine. 
\end{proof}

Proposition~\ref{prop:definitions} shows that the Definition~\cite[2.28]{GrMa09} and  Definition~\ref{def:hom-isop-ineq} of linear homological  isoperimetric inequality are equivalent. 

\begin{proposition}\label{prop:definitions}
A complex $X$ satisfies a {linear homological isoperimetric inequality over $\K$} if and only if there is a constant $A\geq 0$ such that for any circuit $c$ in the $1$-skeleton of $X$  there is $\beta\in C_2(X, \K)$ such that $\partial \beta$ equals  the $1$-cycle induced by $c$ and $\|\beta\|_1 \leq A |c|$.
\end{proposition}
\begin{proof}
The \emph{only if part}  follows from the observation that for a circuit $c$, the $\ell_1$-norm of the induced $1$-cycle and the combinatorial length $|c|$ are equal. For the \emph{if part}, invoking Lemma~\ref{lem:circuitd}, any cycle $\gamma \in Z_1(X, \Z)$ is a finite sum $\sum_{i} \alpha_i$ where each $\alpha_i$ is a $1$-cycle induced by a circuit and $\|\gamma\|_1 = \sum_{i} \|\alpha_i\|_1$.  For this type of expression, we have that $\|\gamma\|_\partial \leq \sum_{i} \|\alpha_i\|_\partial$ from which the implication follows. 
\end{proof}

\subsection{Proof of Theorem~\ref{thm:forfuture}}

\begin{remark}\label{rem:linearity}
Let $X$ be a complex and suppose there is $C\geq 0$ such that $FV_{X, \Z}(k)\leq Ck$ for every $k$.
Then $FV_{X, \Q}(k) \leq Ck$ for every $k$. Indeed, let $\alpha$ by a $1$-cycle in $Z_1(X, \Q)$. Then there is an integer $m$ such that $m\alpha \in Z_1(X, \Z)$.  It follows that there is a $2$-chain $\beta \in C_2(X, \Z)$ such that $\partial \beta = m\alpha$ and $\|\beta\|_1 \leq FV_{X, \Z}(\|m\alpha\|_1) \leq C\|m\alpha\|_1\leq mC\|\alpha\|_1.$  In particular, $ \partial \frac1m \beta =\alpha$ and $\|\frac1m \beta\|_1 \leq C\|\alpha\|_1$. Since $\alpha$ was an arbitrary element, we have $FV_{X, \Q}(k)\leq Ck$. 
\end{remark}

The following lemma uses notation introduced in Lemma~\ref{lem:simply-connected-complex}.

\begin{lemma}\label{lem:omegan2}
Let $\Gamma$ be a connected, fine, hyperbolic graph  equipped with a cocompact $G$-action with finite edge stabilizers.  If $n$ is a large enough integer, then $X=\Omega_n(\Gamma)$ is a simply-connected cocompact $G$-complex such that $FV_{X, \Z}$ is bounded from above by a linear function.
\end{lemma}
\begin{proof}
Invoke Proposition~\ref{lem:simply-connected-complex} to obtain an integer $n$ such that   $X=\Omega_n(\Gamma)$ is a simply-connected complex with $1$-skeleton $\Gamma$ and with  linear isoperimetric function. The $G$-action on $\Gamma$ extends to an action on $X$.  By construction,  the collection of $2$-cells of $X$ are in one-to-one correspondence with circuits in $\Gamma$ of length at most $n$. Since there are finitely many $G$-orbits of $1$-cells in $\Gamma$ and each $1$-cell appears in finitely many circuits of length at most $n$, there are finitely many $G$-orbits of $2$-cells. Remark~\ref{rem:linearclosure} and Lemma~\ref{prop:finiteFV} imply that $FV_{X, \Z}$ is bounded from above by a linear function.
\end{proof}

\begin{proof}[Proof of Theorem~\ref{thm:forfuture}]
By taking a (double) barycentric subdivision of $Y$, assume that attaching maps of $2$-cells of $Y$ are embedded circuits in its $1$-skeleton and there no pairs of $2$-cells with the same boundary path. Observe that taking barycentric subdivisions preserve the hypothesis on $Y$ in view of Lemmas~\ref{lem:subdivision}  and~\ref{lem:gsubdivision}. 

Let $\Gamma$ be the $1$-skeleton of $Y$.  By Lemma~\ref{lem:omegan2}, there is  $n$ such that $X=\Omega_n(\Gamma)$ is a simply-connected cocompact $G$-complex such that $FV_{X, \Z}$ is bounded from above by a linear function. By the assumption on the attaching maps of $2$-cells of $Y$,  taking $n$ large enough implies that $Y$ can be considered as a $G$-equivariant subcomplex of $X$.  

By Theorem~\ref{prop:main}, $FV_{X, \Z}(k)<\infty$ and $FV_{Y, \Z}(k)<\infty$ for every $k$. 
Since $Y$ and $X$ have the same $1$-skeleton, if we let $C=FV_{Y,\Z}(n)$ then  $FV_{Y,\Z}(k) \leq C\cdot FV_{X,\Z}(k)$ for every $k$. Indeed, this follows by observing that any $2$-chain $\mu \in C_2(X)$ can be replaced by a $2$-chain $\nu \in C_2(Y)$ such that $\partial \mu =\partial \nu$ and $\|\nu\|_1 \leq C\|\mu\|_1$.  It follows that $FV_{Y, \Z}$ is also  bounded from above by a linear function.
\end{proof}

 \subsection{Proof of Theorem~\ref{thm:chrhyp}}

 \begin{proof}[Proof of Theorem~\ref{thm:chrhyp}]
Suppose $G$ is hyperbolic relative to $\mc P$.  A complex $X$ with the required properties is obtained by invoking Lemma~\ref{lem:omegan2}.

Conversely,  suppose that there is a complex $X$ with the four properties. By cocompactness there is a bound on the length of attaching maps of $2$-cells. Then Corollary~\ref{cor:suskey} implies that the $1$-skeleton $\Gamma$ of $X$ is a fine hyperbolic graph, and hence  the $G$-action on $\Gamma$ satisfies Definition~\ref{def:BowRH} of relative hyperbolicity.
\end{proof}

\subsection{Coned-off Cayley Complexes and Homological Dehn Functions}\label{sec:3.3}

Let $G$ be a group and let $\mc P$ be a finite collection of finitely generated subgroups.
Suppose there is a finite relative presentation   $\langle S, \mc P |  r_1, \ldots, r_m \rangle$ of $G$   with respect $\mc P$, for a definition see~\cite{Os06}.  Assume  that each $P\in \mc P$ is generated by $S\cap P$, and that $S$ is symmetrized, that is, $S=S^{-1}$. Assume that for each $s\in S$, the relation $ss^{-1}$ is one of the $r_i$'s.  

The \emph{coned-off Cayley graph $\hat \Gamma = \hat \Gamma (G, \mc P, S)$  of $G$ relative to $\mc P$ and $S$} is the $G$-graph obtained from the standard Cayley graph of $G$ with respect to $S$,  by adding  a new (cone) vertex $v(gP)$ for each left coset $gP$ with $g\in G$ and $P\in \mc P$, and edges from $v(gP)$ to each element of $P$. The cone-vertices are in one-to-one correspondence with the collection of left cosets of subgroups in $\mc P$, the $G$-action on the cone-vertices is defined using the corresponding $G$-action on left cosets by $G$.

 The \emph{coned-off Cayley complex $\hat C$} induced by the relative presentation  $\langle S, \mc P |  r_1, \ldots, r_m \rangle$,  it is the $2$-complex obtained by equivariantly attaching $2$-cells to the coned-off Cayley graph $\hat \Gamma$  as follows.  Observe that the relators $r_i$ correspond to loops in $\hat \Gamma$. Attach a $2$-cell with trivial stabilizer to each such loop, and extend in a manner equivariant under the $G$-action on $\hat \Gamma$. Similarly, for each $P\in \mc P$, for each generator in $s\in S\cap P$ and each $g\in G$ corresponds a loop in $\hat \Gamma$ of length three passing through the vertices $g, gs, v(gP)$, where $v(gP)$ is the cone-vertex corresponding to the left coset $gP$. Attach a $2$-cell with trivial stabilizer to each such loop, equivariantly under the $G$-action. This definition of the coned-off Cayley complex appears in~\cite[Definition 2.47]{GrMa09}.
 
The following characterization of relative hyperbolicity in terms of linear homological Dehn functions on coned-off Cayley complexes appears in the work of Groves and Manning~\cite[Theorem 3.25]{GrMa09}. Their proof uses other characterizations of relative hyperbolicity. Below we provide the sketch of a more direct proof of this characterization using the results of this note. 

\begin{theorem}
Let $G$ be a group and let $\mc P$ be a finite collection of finitely generated subgroups. The following statements are equivalent. 
\begin{enumerate}
\item \label{eq1} $G$ is hyperbolic relative to $\mc P$ in the sense of Bowditch, Definition~\ref{def:BowRH}.
\item \label{eq2} $G$ is finitely presented relative to $\mc P$, and for any finite relative presentation $\langle S, \mc P | \mc R \rangle$,  the coned-off Cayley complex $\hat C = \hat C(G, \mc P, S)$  satisfies a linear homological  isoperimetric inequality over the integer numbers.
\item \label{eq3} There is a finite relative presentation $\langle S, \mc P | \mc R \rangle$ such that the corresponding coned-off Cayley complex $\hat C = \hat C(G, \mc P, S)$  satisfies a linear homological  isoperimetric inequality over the integer  numbers.
\item \label{eq4} $G$ is finitely presented relative to $\mc P$, and for any finite relative presentation $\langle S, \mc P | \mc R \rangle$,  the coned-off Cayley complex $\hat C = \hat C(G, \mc P, S)$  satisfies a linear  homological isoperimetric inequality over the rational numbers.
\item \label{eq5} There is a finite relative presentation $\langle S, \mc P | \mc R \rangle$ such that the corresponding coned-off Cayley complex $\hat C = \hat C(G, \mc P, S)$  satisfies a  linear homological isoperimetric inequality over the rational numbers.
\end{enumerate}
\end{theorem} 

\begin{proof}
The implications $\eqref{eq2} \Rightarrow \eqref{eq3}$ and $\eqref{eq4} \Rightarrow \eqref{eq5}$ are trivial. 
The implications $\eqref{eq2} \Rightarrow \eqref{eq4}$ and $\eqref{eq3} \Rightarrow \eqref{eq5}$ follow from the observation that   for a complex $X$, if $FV_{X, \Z}$ is linearly bounded, then $FV_{X, \Q}$ is linearly bounded as well; see Remark~\ref{rem:linearity}. 

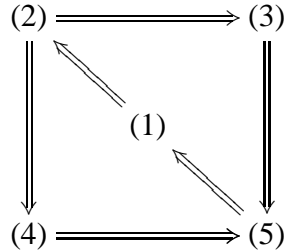
\begin{equation}\nonumber \xymatrix{ 
 \eqref{eq2} \ar@{=>}[dd] \ar@{=>}[rr]  &  &  \eqref{eq3} \ar@{=>}[dd]\\
   &     \eqref{eq1} \ar@{=>}[lu] &  \\
 \eqref{eq4}  \ar@{=>}[rr] \ar@{=>}[rr] &  &  \eqref{eq5} \ar@{=>}[lu]
  }\end{equation}

The implication $\eqref{eq5} \Rightarrow \eqref{eq1}$ is proved as follows.
If $\hat C$ is the coned-off Cayley complex corresponding to a finite relative presentation satisfying a linear homological isoperimetric inequality over the rational numbers, then its $1$-skeleton is a cocompact $G$-graph with trivial edge stabilizers by construction, it is simply-connected~\cite[Lemma 2.48]{GrMa09},  it is fine by Theorem~\ref{thm:answer}\eqref{main2}, and it is hyperbolic by Theorem~\ref{prop:hyperbolicity}. 

The implication $\eqref{eq1} \Rightarrow \eqref{eq2}$ is proved as follows.
Let $\mc K$ be a $(G, \mc P)$-graph, see  Definition~\ref{def:BowRH}. Then Lemma~\ref{lem:omegan2} implies that $\mc K$ is the $1$-skeleton of a simply-connected cocompact $G$-complex, and from here one verifies that $G$ is finitely presented relative to $\mc P$. Let  $\langle S, \mc P | \mc R \rangle$ be an arbitrary finite relative presentation of $G$ with respect to $\mc P$,  let $\hat C$ be the corresponding coned-off Cayley complex, and let $\hat \Gamma$ be its $1$-skeleton.  The a combinatorial construction shows that $\hat \Gamma$ quasi-isometrically embeds as a subgraph of a $(G, \mc P)$-graph; this construction has been studied by different authors, first in Dahmani's thesis~\cite[Proof of Lemma A.4]{Da03}, then in Hruska's work~\cite[Proof of (R-H4) $\Rightarrow$ (RH-5)]{HK08}, and also by Wise and the author of this note~\cite[Proposition 4.3]{MaWi11}. Since hyperbolicity is preserved by quasi-isometry and fineness is preserved by taking subgraphs, it follows that $\hat \Gamma$ is a $(G, \mc P)$-graph. Then Theorem~\ref{thm:forfuture} implies that $\hat C$ satisfies a homological linear isoperimetric inequality over the integers. 
\end{proof}

\section*{Acknowledgements}
We thank Gaelan Hanlon for useful comments on an early draft of this note. We also thank the referees of the article for comments and suggestions. 
The author is funded by the Natural Sciences and Engineering Research Council of Canada, NSERC.

\bibliographystyle{plain}

\begin{thebibliography}{10}

\bibitem{Bo12}
B.~H. Bowditch.
\newblock Relatively hyperbolic groups.
\newblock {\em Internat. J. Algebra Comput.}, 22(3):1250016, 66, 2012.

\bibitem{BrHa99}
Martin~R. Bridson and Andr{\'e} Haefliger.
\newblock {\em Metric spaces of non-positive curvature}, volume 319 of {\em
  Grundlehren der Mathematischen Wissenschaften [Fundamental Principles of
  Mathematical Sciences]}.
\newblock Springer-Verlag, Berlin, 1999.

\bibitem{Da03}
F.~Dahmani.
\newblock {\em Les groupes relativement hyperboliques et leurs bords}.
\newblock PhD thesis, Univ. Louis Pasteur, Strasbourg, France, 2003.

\bibitem{GerstenCohomology}
S.~M. Gersten.
\newblock A cohomological characterisation of hyperbolic groups.
\newblock Available at http://www.math.utah.edu/~sg/Papers/ch.pdf.

\bibitem{Ge99}
S.~M. Gersten.
\newblock Homological dehn functions and the word problem.
\newblock Available at http://www.math.utah.edu/~sg/Papers/df9.pdf.

\bibitem{Ge96}
S.~M. Gersten.
\newblock Subgroups of word hyperbolic groups in dimension {$2$}.
\newblock {\em J. London Math. Soc. (2)}, 54(2):261--283, 1996.

\bibitem{Ge98}
S.~M. Gersten.
\newblock Cohomological lower bounds for isoperimetric functions on groups.
\newblock {\em Topology}, 37(5):1031--1072, 1998.

\bibitem{Gr87}
M.~Gromov.
\newblock Hyperbolic groups.
\newblock In {\em Essays in group theory}, volume~8 of {\em Math. Sci. Res.
  Inst. Publ.}, pages 75--263. Springer, New York, 1987.

\bibitem{GrMa09}
Daniel Groves and Jason~Fox Manning.
\newblock Dehn filling in relatively hyperbolic groups.
\newblock {\em Israel J. Math.}, 168:317--429, 2008.

\bibitem{HK08}
G.~Christopher Hruska.
\newblock Relative hyperbolicity and relative quasiconvexity for countable
  groups.
\newblock {\em Algebr. Geom. Topol.}, 10(3):1807--1856, 2010.

\bibitem{MP15}
Eduardo Mart{\'{\i}}nez-Pedroza.
\newblock Subgroups of relatively hyperbolic groups of bredon cohomological
  dimension 2.
\newblock arXiv:1508.04865.

\bibitem{MaWi11}
Eduardo Mart{\'{\i}}nez-Pedroza and Daniel~T. Wise.
\newblock Local quasiconvexity of groups acting on small cancellation
  complexes.
\newblock {\em J. Pure Appl. Algebra}, 215(10):2396--2405, 2011.

\bibitem{MW11}
Eduardo Mart{\'{\i}}nez-Pedroza and Daniel~T. Wise.
\newblock Relative quasiconvexity using fine hyperbolic graphs.
\newblock {\em Algebr. Geom. Topol.}, 11(1):477--501, 2011.

\bibitem{Mi02}
Igor Mineyev.
\newblock Bounded cohomology characterizes hyperbolic groups.
\newblock {\em Q. J. Math.}, 53(1):59--73, 2002.

\bibitem{MiYa}
Igor Mineyev and Asli Yaman.
\newblock Relative hyperbolicity and bounded cohomology.
\newblock preprint.

\bibitem{Os06}
Denis~V. Osin.
\newblock Relatively hyperbolic groups: intrinsic geometry, algebraic
  properties, and algorithmic problems.
\newblock {\em Mem. Amer. Math. Soc.}, 179(843):vi+100, 2006.

\end{thebibliography}

\end{document}